\documentclass{amsart}
\usepackage{amsmath}

\usepackage[left=25mm,right=25mm,top=25mm,bottom=25mm]{geometry}
\usepackage{graphicx}
\usepackage[
	pdftitle={},
	pdfauthor={},
	ocgcolorlinks,
	linkcolor=linkblue,
	citecolor=linkred,
	urlcolor=linkred]
{hyperref}
\usepackage{color}
\definecolor{linkred}{rgb}{0.6,0,0}
\definecolor{linkblue}{rgb}{0,0,0.6}

\usepackage{mathpazo}
\usepackage{microtype}
\usepackage{multicol}
\usepackage{booktabs}

\usepackage{multirow}
\usepackage{setspace}

\theoremstyle{plain}
\newtheorem{theorem}{Theorem}
\newtheorem{lemma}{Lemma}

\newtheorem{corollary}[theorem]{Corollary}

\theoremstyle{definition}

\newcommand{\bc}{\mathbb{C}}

\newcommand{\bp}{\mathbb{P}}
\newcommand{\bq}{\mathbb{Q}}
\newcommand{\bz}{\mathbb{Z}}

\newcommand{\modm}{\mathcal{M}}

\newcommand{\res}[1]{\begin{array}[d]{l}\\{\rm Res}\\^{#1}\end{array}\hspace{-1mm}}

\newtheorem{defn}{Definition}

\begin{document}

\title{Stationary Gromov-Witten invariants of projective spaces}
\author{Paul Norbury}
\address{Department of Mathematics and Statistics, The University of Melbourne, Victoria 3010, Australia}
\email{\href{mailto:norbury@unimelb.edu.au}{norbury@unimelb.edu.au}}

\thanks{This work was supported by the Australian Research Council.}
\date{}
\begin{abstract}
We represent stationary descendant Gromov-Witten invariants of projective space, up to explicit combinatorial factors, by polynomials.  One application gives the asymptotic behaviour of large degree behaviour of stationary descendant Gromov-Witten invariants in terms of intersection numbers over the moduli space of curves.  We also show that primary Gromov-Witten invariants are "virtual" stationary descendants and hence the string and divisor equations can be understood purely in terms of stationary invariants.
\end{abstract}

\maketitle

\tableofcontents

\section{Introduction} \label{introduction}

Let $X$ be a projective algebraic variety and $(C,x_1,\dots,x_n)$ a connected smooth curve of genus $g$ with $n$ distinct marked points.  For $\beta \in H_2(X,\bz)$ the moduli space of maps $\modm^g_n(X,\beta)$ consists of morphisms 
$$\pi: (C,x_1,\dots,x_n)\rightarrow X$$
satisfying $\pi_\ast [C]=\beta$ quotiented by isomorphisms of the domain $C$ that fix each $x_i$.  The moduli space has a compactification $\overline{\modm}^g_n(X,\beta)$ given by the moduli space of stable maps:  the domain $C$ is a connected nodal curve; the distinct points $\{x_1,\dots,x_n\}$ avoid the nodes; any genus zero irreducible component of $C$ with fewer than three distinguished points (nodal or marked) must not be collapsed to a point; any genus one irreducible component of $C$ with no marked point must not be collapsed to a point.  The moduli space of stable maps may have irreducible components of different dimensions but its expected or virtual dimension is 
\begin{equation}   \label{eq:vdim}
\dim\overline{\modm}^g_n(X,\beta)=\langle c_1(X),\beta\rangle +(\dim X-3)(1-g)+n.
\end{equation}
Any cohomology class $\gamma\in H^*(X,\bz)$ pulls back to a cohomology class $ev^\ast_i(\gamma)$ in $H^*(\overline{\modm}^g_n(X,\beta),\bq)$ via the evaluation map $ev_i:\overline{\modm}^g_n(X,\beta)\longrightarrow X, ~ ev_i(\pi)=\pi(x_i)$ for $i=1,\dots,n$.  Further cohomology classes $\psi_i\in H^2(\overline{\modm}^g_n(X,\beta),\bq)$ are obtained from the first Chern class of the tautological line bundle $\mathcal{L}_i$ over $\overline{\modm}^g_n(X,\beta)$ with fibre given by the cotangent bundle of $T^*_{x_i}C$ over the $i$th marked point.

Define the descendant Gromov-Witten invariants of $X$ by:
\begin{equation}\label{eq:gropoint1}
\left\langle \prod_{i=1}^n\tau_{m_i}(\gamma_i)\right\rangle ^g_{X,\beta}=\int_{[\overline{\modm}^g_n(X,\beta)]^{vir}} \prod_{i=1}^n\psi_i^{m_i}ev_i^\ast(\gamma_i).
\end{equation}
The integration is against the \textit{virtual fundamental class}, $[\overline{\modm}^g_n(X,\beta)]^{vir}$ and (\ref{eq:gropoint1}) is defined to be zero unless $\sum_{i=1}^n m_i+\deg\gamma_i=\langle c_1(X),\beta\rangle +(\dim X-3)(1-g)+n$.  We may drop $g$, $d$ or $X$ from the notation when it is understood.

This paper will be principally concerned with {\em primary} Gromov-Witten insertions $\tau_0(\gamma_i)$  where $m_i=0$ and $\gamma_i$ is arbitrary and {\em stationary} insertions $\tau_m(\gamma_i)$ where $\gamma_i$ is Poincare dual to a point.   For a stationary insertion we usually write $\tau_{m_i}(pt)$ in place of $\tau_m(\gamma_i)$ to emphasise that the $i$th point is stationary, i.e. it must map to a given point in $X$.  

Restrict to $X=\bp^N$ for $N>0$ and let $\omega\in H^2(\bp^N,\bq)$ be the generator of $H^*(\bp^N,\bq)$ so $\omega^N$ is the Poincare dual class of a point.  The degree of a map $C\to\bp^N$ is simply an integer $d\in\bz\cong H_2(\bp^N)$.  Also $c_1(\bp^N)=(N+1)\omega$ hence $\dim\overline{\modm}^g_n(\bp^N,d)=(N+1)d+(N-3)(1-g)+n$ and this gives rise to $\bmod{\ N+1}$ dependence of the invariants.

Primary Gromov-Witten invariants and 1-point descendant Gromov-Witten invariants are  fundamental via various reconstruction theorems for Gromov-Witten invariants \cite{GatTop,GivSem,KMaGro}.  In this paper we take a different point of view and show that the {\em stationary} Gromov-Witten invariants are somehow fundamental and particularly well-behaved.  They have a nice polynomial form which allows closed form expressions and they satisfy recursions (without using non-stationary Gromov-Witten invariants.) They take the position as the fundamental invariants since the primary invariants can be represented as virtual stationary invariants, 
\[ \tau_0(\omega^k)="\tau_{k-N}(pt)",\quad k=0,...,N.\]
The negative stationary insertion is explained in Theorem~\ref{th:negeval}.  With this viewpoint, the divisor and string equations, which usually require a non-stationary term, become relations between stationary invariants alone.

Given $m,N\in\{1,2,3,...\}$, define 
\[ c_N(m)=\left\lceil m/N\right\rceil\cdot c_N(m-1),\quad c_N(0)=1\]
where the ceiling function $\left\lceil r \right\rceil$ is the smallest integer not less than $r$. 
So $c_N(m)$  generalises $m!=c_1(m)$.  For $m>0$, an explicit formula is
$c_N(m)=\left\lceil m/N\right\rceil !^N\left\lceil m/N\right\rceil^{m-N\lceil m/N\rceil}$.

The stationary Gromov-Witten invariants of $\bp^N$ have polynomial behaviour as follows.
\begin{theorem} \label{th:GWquasi}
For $2g-2+n>0$ and $m_i\geq 3g-1$, $i\in\{1,...,n\}$ define
\begin{equation}\label{eq:quasi}
p^{(N)}_g(m_1,\dots,m_n):=\left\langle \prod_{i=1}^n\tau_{m_i}(pt)\right\rangle^g_{\bp^N}\hspace{-.2cm}\cdot\hspace{.2cm}\prod_{i=1}^nc_{N+1}(m_i).
\end{equation} 
Then $p^{(N)}_g(m_1,\dots,m_n)$ is a degree $3g-3+n$ symmetric quasi-polynomial, in the sense that it is polynomial on each coset of the sublattice $(N+1)\bz^n \subset \bz^n$.  The top coefficients $c^{(N)}_{\beta}$ of $m_1^{\beta_1}\cdots m_n^{\beta_n}$ are given by 
\begin{equation}  \label{eq:gwcoeff}
c^{(N)}_{\beta}=(N+1)^{3-2g-n}\int_{\overline{\modm}_{g,n}}\psi_1^{\beta_1}...\psi_n^{\beta_n}
\end{equation}
for $|\beta|=3g-3+n$.
\end{theorem}
{\em Remark.}  We expect to be able to drop the restriction $m_i\geq 3g-1$.  This is true for $g=0$ since $m_i$ are necessarily non-negative and the theorem can be strengthened so that for $g=1$ the restriction $m_i\geq 3g-1$ can also be dropped.  For $g=2$ it can be relaxed to $m_i\geq 2$.  

By the dimension constraint, $p^{(N)}_g(m_1,\dots,m_n)$ is non-trivial only when $\sum m_i\equiv2(2g-2+n)\bmod{N+1}$.

The genus zero 1 and 2-point functions $p^0_1(m)$ and $p^0_2(m_1,m_2)$ defined analogously to (\ref{eq:quasi}) in Section~\ref{sec:gwpol} can be thought of as degree -2 and -1 quasi-polynomials, respectively.
\begin{corollary}  \label{th:asym}
For $2g-2+n>0$, the stationary Gromov-Witten invariants of $\bp^N$ behave asymptotically as
\begin{equation}\label{eq:asym}
\left\langle \prod_{i=1}^n\tau_{m_i}(pt)\right\rangle^g_{\bp^N}
\hspace{-.1cm}\sim\hspace{.2cm}
 \frac{(N+1)^{3-2g-n}}{\prod_{i=1}^n{c_{N+1}(m_i)}}\sum_{|\beta|=3g-3+n}m_1^{\beta_1}\cdots m_n^{\beta_n}\int_{\overline{\modm}_{g,n}}\psi_1^{\beta_1}...\psi_n^{\beta_n}.
\end{equation}
\end{corollary}

The quasi-polynomial $p^{(N)}_g(m_1,...,m_n)$ only makes enumerative sense when its entries satisfy $m_i\geq 3g-1$.  As mentioned above, we expect evaluation at $0\leq m_i<3g-1$ to give the expected stationary invariants.  Furthermore, evaluation at the negative integers $k-N$ for $k=0,...,N-1$, makes sense and one can give an enumerative interpretation as follows.
\begin{theorem}\label{th:negeval}For $2g-2+n\geq 0$  
\[ \tau_0(\omega^k)="\tau_{k-N}(pt)",\quad k=0,...,N\]
via evaluation of the quasi-polynomial $p^{(N)}_g(m_1,...,m_n)$ at negative integers.  More precisely: 
\begin{equation}\label{eq:evalneg}
\left\langle\prod_{j=1}^s\tau_0(\omega^{k_j}) \prod_{i=1}^n\tau_{m_i}(pt)\right\rangle^g_{\bp^N}\hspace{-.2cm}\cdot\hspace{.2cm}\prod_{i=1}^nc_{N+1}(m_i)=
p^{(N)}_g(k_1-N,...,k_s-N,m_1,...,m_n).
\end{equation} 
\end{theorem}
The Gromov-Witten invariants are not quasi-polynomial in non-stationary descendant variables.  A counter-example is given in Section~\ref{sec:gwpol}.  

The divisor and string equations \cite{WitTwo} satisfied by stationary Gromov-Witten invariants are:  
\begin{align*}
\left\langle \tau_0(\omega)\prod_{i=1}^n\tau_{m_i}(pt)\right\rangle^g_d&=d\left\langle \prod_{i=1}^n\tau_{m_i}(pt)\right\rangle^g_d&{\rm divisor\ equation}\\
\left\langle \tau_0(1)\prod_{i=1}^n\tau_{m_i}(pt)\right\rangle^g_d&=\sum_{i=1}^n\left\langle \tau_{m_1}(pt)\cdots\tau_{m_i-1}(pt)\cdots \tau_{m_n}(pt)\right\rangle^g_d&{\rm string\ equation}
\end{align*}
where the term $\tau_{m_i-1}(pt)$ vanishes if $m_i=0$.  They necessarily involve the non-stationary terms $\tau_0(\omega)$ and $\tau_0(1)$.  Using Theorem~\ref{th:negeval} to interpret $\tau_0(\omega)=\tau_{1-N}(pt)$ and $\tau_0(1)=\tau_{-N}(pt)$ we get relations involving only stationary terms.
\begin{corollary}
The divisor and string equations can be expressed entirely in terms of stationary invariants.
\end{corollary}

The $N=0$ and $N=1$ cases are both interesting and important.  The $N=0$ case does not correspond directly to target space the point $\bp^0$ which admits only degree 0 maps.  Rather, we introduce a degree $d$ by allowing $d$ unlabeled points on the domain curves.  This case is important because it is used to calculate the top degree terms of $p^{(N)}_g(m_1,\dots,m_n)$ which gives the large degree asymptotic behaviour of the Gromov-Witten invariants of $\bp^N$.  It also shows that the top degree terms of $p^{(N)}_g(m_1,\dots,m_n)$ are in fact polynomial rather than quasi-polynomial.  The $N=1$ case was studied in \cite{NScGro} where the quasi-polynomial behaviour and large degree asymptotic behaviour was conjectured.  Understanding this is part of the motivation for this paper.

Following \cite{NScGro}, assemble the (connected) stationary Gromov-Witten invariants of $\bp^1$ into the generating function multidifferential
\[\Omega^g_n(x_1,...,x_n)=\sum_{\vec{m}}\left\langle \prod_{i=1}^n \tau_{m_i}(pt) \right\rangle^g_{\bp^1}\hspace{-.2cm}\cdot\hspace{.2cm}\prod_{i=1}^n(m_i+1)!x_i^{-m_i-2}dx_i.\]
\begin{theorem}  \label{th:p1gen}
For $2g-2+n>0$, $\Omega^g_n(x_1,...,x_n)$ is analytic around $x_i=\infty$ and extends to a meromorphic multidifferential on the compact Riemann surface defined by $x=z+1/z$.  It has poles at $z_i=\pm 1$ of order $6g-4+2n$ in each variable and asymptotic behaviour
\[
\Omega^g_n\sim s^{6-6g-3n}\frac{1}{2^{5g-5+2n}}\sum_{|\beta|}\prod_{i=1}^n \frac{(2\beta_i+1)!}{\beta_i!}\frac{dt_i}{t_i^{2\beta_i+2}}\langle \tau_{\beta_1}...\tau_{\beta_n}\rangle_g, 
\]
for the local variable $z_i=\pm1+s\cdot t_i$.
\end{theorem}
In particular, Theorem~\ref{th:p1gen} proves that the generating function
$\Omega^g_n(x_1,...,x_n)$ for stationary Gromov-Witten invariants of $\bp^1$ is {\em algebraic} and moreover rational.  This was conjectured in \cite{NScGro} along with the asymptotic behaviour at the poles and proven there for $g=0$ and 1 as part of the stronger result that the multidifferentials can be defined in another way: the genus 0 and 1 Gromov-Witten invariants of $\bp^1$ coincide with the Eynard-Orantin invariants \cite{EOrInv,EOrTop} of a particular Riemann surface and are recursively calculable. Conjecturally the stronger result holds for general $g$ generating function multidifferentials. In particular, the generating function multidifferentials for stationary Gromov-Witten invariants of $\bp^1$ are conjecturally known.  For example, the genus two 1-point generating function differential has been checked numerically to coincide with the known differential.  Moreover, the quasi-polynomial applies to all $m_i$ not just those satisfying $m_i\geq 3g-1$.  See Section~\ref{sec:EO} for more details.

Section~\ref{sec:gwpol} contains the topological recursion relations satisfied quite generally by Gromov-Witten invariants, and the $N=0$ case which consists of intersection theory on the moduli space of curves.  These are used to prove Theorem~\ref{th:GWquasi} , Corollary~\ref{th:asym} and Theorem~\ref{th:negeval}.  In Section~\ref{sec:P1} we specialise to the $N=1$ case and study the string and divisor equations for general $N$.  Explicit formulae appear in Section~\ref{sec:ex}.

{\em Acknowledgements.}  I would like to thank the Department of Mathematics at LMU, Munich for its hospitality during the second half of 2011 during which this research was carried out.

\section{Polynomial behaviour of Gromov-Witten invariants}  \label{sec:gwpol}

We will need the following recursion relations satisfied by Gromov-Witten invariants of any target space.  For $\gamma_i\in H^*(X)$,
\begin{align}
\left\langle \tau_0(1)\prod_{i=1}^n\tau_{m_i}(\gamma_i)\right\rangle^g_d&=\sum_{i=1}^n\left\langle \tau_{m_1}(\gamma_1)\cdots\tau_{m_i-1}(\gamma_i)\cdots \tau_{m_n}(\gamma_n)\right\rangle^g_d&{\rm string\ equation}\ \ \\
\left\langle \tau_0(\omega)\prod_{i=1}^n\tau_{m_i}(\gamma_i)\right\rangle^g_d&=d\left\langle \tau_{m_1}(\gamma_1)\cdots \tau_{m_n}(\gamma_n)\right\rangle^g_d\label{GWdivisor}&{\rm divisor\ equation}\\
&\quad+\sum_{i=1}^n\left\langle \tau_{m_1}(\gamma_1)\cdots\tau_{m_i-1}(\gamma_i\cup\omega)\cdots \tau_{m_n}(\gamma_n)\right\rangle^g_d\nonumber&
\end{align}
where $d$ satisfies $(N+1)d+(N-3)(1-g)+n=\sum_{i=1}^n( m_i+\deg\gamma_i)$.

The following standard notation is used for expressing the topological recursion relations  below.
\[
\left\langle\left\langle\prod_{i=1}^n\tau_{m_i}(\gamma_i)\right\rangle\right\rangle^g=
\left\langle\prod_{i=1}^n\tau_{m_i}(\gamma_i)\exp\sum_{k,r}t_k^r\tau_k(\gamma_r)\right\rangle^g
=\sum_{s=0}^{\infty}\frac{1}{s!}\sum_{\vec{k},\vec{r}}\prod_{j=1}^st_{k_j}^{r_j}\left\langle\prod_{i=1}^n\tau_{m_i}(\gamma_i)\prod_{j=1}^s\tau_{k_j}(\gamma_{r_j})\right\rangle^g.
\]
Let $\{T_j\}$ be a basis for $H^*(X)$ and let $\{T^j\}\subset H^*(X)$ be the dual basis obtained via Poincare duality.  Inside intersection brackets, we often identify $T_j$ with its pull-back to the moduli space, i.e. we write $T_j$ for $\tau_0(T_j)$.  We follow the usual convention of summing over any index that appears twice in a formula as a subscript and superscript.

For $\gamma_i\in H^*(X)$ a genus zero topological recursion (TRR) is \cite{WitTwo}
\begin{equation}  \label{eq:TRR0}
\left\langle\left\langle \tau_{m_1+1}(\gamma_1) \tau_{m_2}(\gamma_2) \tau_{m_3}(\gamma_3) \right\rangle\right\rangle^0=\sum_j\left\langle\left\langle \tau_{m_1}(\gamma_1)T_j\right\rangle\right\rangle^0 \left\langle\left\langle T^j\tau_{m_2}(\gamma_2)\tau_{m_3}(\gamma_3)\right\rangle\right\rangle^0.
\end{equation}
For any $g>0$, a topological recursion relation \cite{EXiQua,GatGro,LiuQua} is:
\begin{equation}   \label{eq:TRRg}
\left\langle\left\langle \tau_{m+3g-1}(\gamma)\right\rangle\right\rangle^g
=\sum_{\alpha+\beta=3g-2}\left\langle\left\langle \tau_{\alpha}(T_k)\right\rangle\right\rangle^g
\left\langle\left\langle T^k\tau_m(\gamma)\right\rangle\right\rangle_{(\beta)}
\end{equation}
where the function $\left\langle\left\langle T^j\tau_m(\gamma)\right\rangle\right\rangle_{(\beta)}$ involves only genus zero invariants.  It is defined by
\[ \left\langle\left\langle T^j\tau_m(\gamma)\right\rangle\right\rangle_{(\beta)}=\sum_{k=1}^{3g-1}(-1)^{k-1}\hspace{-1.2cm}\sum_{\begin{array}{c}(m_1,...,m_k)\\k+\displaystyle{\sum_{i>0} m_i}=\beta+1\end{array}} \hspace{-1cm} 
\left\langle\left\langle T_{a_k}\tau_{m_k+m}(\gamma) \right\rangle\right\rangle^0
\prod_{i=1}^{k-1}
\left\langle\left\langle T_{a_i}\tau_{m_i}(T^{a_{i+1}}) \right\rangle\right\rangle^0,\quad a_1=j
\]
and can also be obtained recursively by
\[ \left\langle\left\langle T^j\tau_m(\gamma)\right\rangle\right\rangle_{(\beta)}=
\left\langle\left\langle T^j\tau_{m+1}(\gamma)\right\rangle\right\rangle_{(\beta-1)}
-\left\langle\left\langle T^i\tau_m(\gamma)\right\rangle\right\rangle_0\left\langle\left\langle T_iT^j\right\rangle\right\rangle_{(\beta-1)}\]
with initial condition $\left\langle\left\langle \cdots\right\rangle\right\rangle_{(0)}=\left\langle\left\langle \cdots\right\rangle\right\rangle^{0}$.

For $X=\bp^N$, we take $T_j=\omega^j$, $T^j=\omega^{N-j}$.  Note that in this case the dimension constraint uniquely chooses a $T_j$ in each bracket so that the sum over $T_jT^j$ consists of a single term.\\

We prove Theorem~\ref{th:GWquasi} by induction.  First we calculate the genus zero 1-point and 2-point functions which are required as initial conditions in the induction.

{\em Genus zero one-point and two-point descendant invariants.}  

The 1-point genus zero stationary invariants can be determined via the genus zero topological recursion relation (\ref{eq:TRR0}).  They are
\[ \langle\tau_{(N+1)d-2}(pt)\rangle^0=\frac{1}{d!^{N+1}}\]
or equivalently 
\[ \left\langle\tau_m(pt)\right\rangle^{0}=\frac{1}{ c_{N+1}(m)}\cdot\frac{1}{d^2},\quad d=\frac{m+2}{N+1}\]
for $m+2\equiv 0\bmod{N+1}$ and 0 otherwise.\\

For the 2-point invariants we need to also allow non-stationary terms.  

\begin{lemma}  \label{th:2primdes}
The genus zero two-point stationary descendant invariants are given by
\[\langle\tau_{m_1}(pt)\tau_{m_2}(pt)\rangle=\frac{1}{c_{N+1}(m_1)c_{N+1}(m_2)}\cdot \frac{1}{d},\quad d=1+\frac{m_1+m_2}{N+1}\] 
and primary and stationary insertions together are given by
\[\langle\tau_{m}(pt)\tau_0(\omega^k)\rangle=\frac{1}{c_{N+1}(m)}\cdot \frac{1}{d}=\frac{1}{c_{N+1}(m+1)},\quad d=\frac{m+k+1}{N+1}\]
for $m_1+m_2\equiv 0\bmod{N+1}$ and $m+k+1\equiv 0\bmod{N+1}$ and they vanish otherwise.
\end{lemma}
Note that genus zero two-point primary invariants vanish because the stable maps necessarily have degree zero and there are no genus zero 2-pointed degree zero stable maps.
\begin{proof}
The proof is by induction.  We first prove the case involving primary insertions by induction on $k$.  The initial case in the induction is the stationary case $k=N$, so $m\equiv 0\bmod{N+1}$.  Note that in each case below the degree is positive and hence we can apply the divisor equation.
\begin{align*}
\langle\tau_{m}(pt)\tau_0(\omega^N)\rangle_d=\langle\tau_{m}(pt)\tau_0(pt)\rangle_d&=\frac{1}{d}\langle\tau_{m}(pt)\tau_0(pt)\tau_0(\omega)\rangle_d&{\rm divisor\ equation}\\
&=\frac{1}{d}\langle\tau_{m-1}(pt)\tau_0(1)\rangle_{d-1}\langle\tau_0(pt)\tau_0(pt)\tau_0(\omega)\rangle_1&{\rm TRR\ }(\ref{eq:TRR0})\\
&=\frac{1}{d}\langle\tau_{m-1}(pt)\tau_0(1)\rangle_{d-1}&\\
&=\frac{1}{d}\langle\tau_{m-2}(pt)\rangle_{d-1}&{\rm string\ equation}\\
&=\frac{1}{d}\cdot\frac{1}{c_{N+1}(m-2)}\cdot\frac{1}{(d-1)^2}&\\
&=\frac{1}{d}\cdot\frac{1}{c_{N+1}(m)}&\hspace{-2cm}\frac{c_{N+1}(m)}{c_{N+1}(m-2)}=\left\lceil\frac{m}{N+1}\right\rceil\left\lceil\frac{m-1}{N+1}\right\rceil 
\end{align*}
For $k<N$ and $m+k+1\equiv 0\bmod{N+1}$,
\begin{align*}
\langle\tau_{m}(pt)\tau_0(\omega^k)\rangle&=\frac{1}{d}\langle\tau_{m}(pt)\tau_0(\omega^k)\tau_0(\omega)\rangle&{\rm divisor\ equation}\\
&=\frac{1}{d}\langle\tau_{m-1}(pt)\tau_0(\omega^{k+1})\rangle\langle\tau_0(\omega^{N-k-1})\tau_0(\omega^k)\tau_0(\omega)\rangle&{\rm TRR\ }(\ref{eq:TRR0})\\
&=\frac{1}{d}\langle\tau_{m-1}(pt)\tau_0(\omega^{k+1})\rangle&\\
&=\frac{1}{d}\cdot\frac{1}{c_{N+1}(m-1)}\cdot\frac{1}{d}&{\rm inductive\ hypothesis}\\
&=\frac{1}{d}\cdot\frac{1}{c_{N+1}(m)}&\hspace{-2cm}c_{N+1}(m)=\left\lceil\frac{m}{N+1}\right\rceil c_{N+1}(m-1)
\end{align*}
as required.  The stationary case:
\begin{align*}
\langle\tau_{m_1}(pt)\tau_{m_2}(pt)\rangle_d&=\frac{1}{d}\langle\tau_{m_1}(pt)\tau_{m_2}(pt)\tau_0(\omega)\rangle_d&{\rm divisor\ equation}\\
&=\frac{1}{d}\langle\tau_{m_1-1}(pt)\tau_0(\omega^k)\rangle_{d_1}\langle\tau_0(\omega^{N-k})\tau_{m_2}(pt)\tau_0(\omega)\rangle_{d_2}&{\rm TRR\ }(\ref{eq:TRR0})\\
&=\frac{d_2}{d}\langle\tau_{m_1-1}(pt)\tau_0(\omega^k)\rangle_{d_1}\langle\tau_0(\omega^{N-k})\tau_{m_2}(pt)\rangle_{d_2}&{\rm divisor\ equation}\\
&=\frac{d_2}{d}\cdot\frac{1}{c_{N+1}(m_1-1)}\cdot\frac{1}{d_1}\cdot\frac{1}{c_{N+1}(m_2)}\cdot\frac{1}{d_2}&{\rm inductive\ hypothesis}\\
&=\frac{1}{d}\cdot\frac{1}{c_{N+1}(m_1)}\cdot\frac{1}{c_{N+1}(m_2)}&\frac{c_{N+1}(m_1)}{c_{N+1}(m_1-1)}=\left\lceil\frac{m_1}{N+1}\right\rceil 
\end{align*}
As already mentioned, in each application of the TRR, the choice of $\omega^j$ is uniquely determined by dimension constraints.
\end{proof}

The first part of Theorem~\ref{th:GWquasi} is an immediate consequence of the following more general statement in which the stationary insertions vary while all others are held fixed. 
\begin{theorem} \label{th:GWquasi1}For $2g-2+n+s>0$ and $m_i\geq 3g-1$, $i=1,...,n$ define
\begin{equation}\label{eq:quasi1}
q^{(N)}_{g,\vec{\kappa}}(m_1,\dots,m_n):=\left\langle \prod_{j=1}^s\tau_{\kappa_j}(\omega^{k_j})\prod_{i=1}^n\tau_{m_i}(pt)\right\rangle^g\cdot\prod_{i=1}^nc_N(m_i)
\end{equation} 
for $\vec{\kappa}=(\kappa_1,k_1,\kappa_2,...,k_s)$.  Then $q^{(N)}_{g,\vec{\kappa}}(m_1,\dots,m_n)$ is a degree $3g-3+n+s$ symmetric quasi-polynomial in $m_i$. 
\end{theorem}
\begin{proof}
When $n=0$ 
\[ q^{(N)}_{g,\vec{\kappa}}:=\left\langle \prod_{j=1}^s\tau_{\kappa_j}(\omega^{k_j})\right\rangle^g\]
is a constant and there is nothing to prove.  This is the initial case in an induction on the genus, the number of insertions and the number of stationary insertions, respectively.  The main tool is the topological recursion relations for Gromov-Witten invariants.  

The genus zero 1-point and 2-point functions described above allow us to define
\[ p^0_2(m_1,m_2):=\frac{N+1}{m_1+m_2+N+1},\quad p^0_1(m):=\frac{(N+1)^2}{(m+2)^2}.\]
Notice that these also have the form (\ref{eq:quasi1}) if we allow degree -2 and -1 "polynomials".  

If there are no descendant stationary terms, i.e. $m_i=0$ for all $i$ then as already mentioned there is nothing to prove since both sides are constant.  So we may assume that $m_1$, say, is non-zero.  

{\bf Genus 0.}  In the genus 0 case $n+s\geq 3$ so we write
\[ \prod_{j=1}^s\tau_{\kappa_j}(\omega^{k_j})\prod_{i=1}^n\tau_{m_i}(pt)=\tau_{m_1}(pt)\tau_{a_2}(\omega^{b_2})\tau_{a_3}(\omega^{b_3})\tau_U\tau_V\]
where $\tau_{a_i}(\omega^{b_i})$ are two factors chosen from $\prod_{j=1}^s\tau_0(\omega^{k_j})\prod_{i\neq 1}\tau_{m_i}(pt)$, and $\tau_U\tau_V$ contains all other factors.

Apply the genus zero TRR (\ref{eq:TRR0})
\begin{equation}  \label{eq:mainduc}
\left\langle \prod_{j=1}^s\tau_{\kappa_j}(\omega^{k_j})\prod_{i=1}^n\tau_{m_i}(pt)\right\rangle
=\sum_{U}\left\langle\tau_0(\omega^{k_U})\tau_{m_1-1}(pt)\tau_U\right\rangle\left\langle\tau_0(\omega^{N-k_U})\tau_{a_2}(\omega^{b_2})\tau_{a_3}(\omega^{b_3})\tau_V\right\rangle
\end{equation}
where, as usual, the choice of $U$ (and hence $V$) uniquely determines $k_U$.  

Each term in the right hand side of (\ref{eq:mainduc}) is simpler in the induction---either there are fewer than $n+s$ insertions in each factor, or there is the term
\begin{equation} \label{eq:2pterm}
\left\langle\tau_0(\omega^k)\tau_{m_1-1}(pt)\right\rangle\left\langle\tau_0(\omega^{N-k})\prod_{j=1}^s\tau_{\kappa_j}(\omega^{k_j})\prod_{i\neq 1}\tau_{m_i}(pt)\right\rangle
\end{equation}
which has  $n+s$ insertions in the second factor, though with only $n-1$ stationary descendants.  

The initial cases consist either of no stationary descendants, where the theorem trivially holds, or the genus zero two-point function whose formula is given in Lemma~\ref{th:2primdes}.

Hence, by induction we can assume that each term is of the form (\ref{eq:quasi1}), where we allow the degree -2 and -1 "polynomials" discussed above.  

By the inductive assumption, any summand of (\ref{eq:mainduc}) satisfies
\[\left\langle\tau_0(\omega^{k_U})\tau_{m_1-1}(pt)\tau_U\right\rangle\left\langle\tau_0(\omega^{N-k_U})\tau_{a_2}(\omega^{b_2})\tau_{a_3}(\omega^{b_3})\tau_V\right\rangle=
\prod_{i=1}^n\frac{1}{c_{N+1}(m_i)} \cdot\left\lceil\frac{m_1}{N+1}\right\rceil q^0_{\vec{\kappa_1},|I_1|}(m_{I_1})q^0_{\vec{\kappa_2},|I_2|}(m_{I_2})\]
where $\#\vec{\kappa_1}+\#\vec{\kappa_2}=s+2$ (for $\#\vec{\kappa_i}$ (half) the number of components of $\vec{\kappa_i}$), $I_1\sqcup I_2=\{1,...,n\}$ and we have used $1/c_{N+1}(m_1-1)=\left\lceil m_1/(N+1)\right\rceil/c_{N+1}(m_1)$.  

If neither factor is a genus zero two-point function then the three factors $\left\lceil\frac{m_1}{N+1}\right\rceil q^0_{\vec{\kappa_1},|I_1|}(m_{I_1})q^0_{\vec{\kappa_2},|I_2|}(m_{I_2})$ are quasi-polynomials of respective degrees 1, $|I_1|+\#\vec{\kappa_1}-3$, $|I_2|+\#\vec{\kappa_2}-3$ so their product is a quasi-polynomial of degree 
\[1+|I_1|+\#\vec{\kappa_1}-3+|I_2|+\#\vec{\kappa_2}-3=n+s-3\]
as required.

The summand (\ref{eq:2pterm}) requires special consideration.  Apply Lemma~\ref{th:2primdes} to the two-point factor $\left\langle\tau_0(\omega^k)\tau_{m_1-1}(pt)\right\rangle$ so $d=\frac{m_1+k}{N+1}=\left\lceil\frac{m_1}{N+1}\right\rceil$ and
\[\left\langle\tau_0(\omega^k)\tau_{m_1-1}(pt)\right\rangle=
\frac{1}{c_{N+1}(m_1-1)}\cdot \frac{1}{d}=\frac{\left\lceil\frac{m_1}{N+1}\right\rceil}{c_{N+1}(m_1)}\cdot \frac{1}{d}=
\frac{1}{c_{N+1}(m_1)}.\]
The factor $\left\langle\tau_0(\omega^{N-k})\prod_{j=1}^s\tau_{\kappa_j}(\omega^{k_j})\prod_{i\neq 1}\tau_{m_i}(pt)\right\rangle=\prod_{i\neq 1}\frac{1}{c_{N+1}(m_i)}\times$ degree $n+s-3$ quasi-polynomial so the summand has the required form and the theorem is proven for genus 0.

{\bf Genus $\bf g>0$.}
We assume $m_1\geq 3g-1$ and write 
\[ \prod_{j=1}^s\tau_{\kappa_j}(\omega^{k_j})\prod_{i=1}^n\tau_{m_i}(pt)=\tau_{m_1}(pt)\tau_U\tau_V\]
where we will sum over all factorisations $\tau_U\tau_V=\prod_{j=1}^s\tau_{\kappa_j}(\omega^{k_j})\prod_{i\neq 1}\tau_{m_i}(pt)$.

Apply the genus $g$ TRR (\ref{eq:TRRg})
\begin{equation} \label{eq:mainduc1}
\left\langle \prod_{j=1}^s\tau_0(\omega^{k_j})\prod_{i=1}^n\tau_{m_i}(pt)\right\rangle^g
=\sum_{\alpha+\beta=3g-2}\sum_{U}\left\langle\tau_0(\omega^{k_U})\tau_{m_1+1-3g}(pt)\tau_U\right\rangle_{(\beta)}\left\langle\tau_{\alpha}(\omega^{N-k_U})\tau_V\right\rangle^g
\end{equation}
Each term in the right hand side of (\ref{eq:mainduc1}) is simpler in the induction---either it is genus 0 or there are fewer than $n+s$ insertions in each factor, or there are fewer than $n$ descendant insertions.

The term $\left\langle\tau_0(\omega^{k_U})\tau_{m_1+1-3g}(pt)\tau_U\right\rangle_{(\beta)}$ consists of only genus zero invariants.  It contains a 2-point term involving $\tau_{m_1+1-3g+\beta}(pt)$ which determines the degree of the following quasi-polynomial----pull out the quotients of $c_{N+1}(m_1)$, as in the genus zero argument above, to get a quasi-polynomial $q_{\vec{\kappa_1,(\beta)},|I_1|}(m_{I_1})$ of degree $-3+|I_1|+\#\vec{\kappa_1}+1-3g+\beta+1\leq -3+|I_1|+\#\vec{\kappa_1}$ since $\beta\leq 3g-2$.

If all $m_i\geq 3g-1$, by induction, each summand $\left\langle\tau_0(\omega^{k_U})\tau_{m_1+1-3g}(pt)\tau_U\right\rangle_{(\beta)}\left\langle\tau_{\alpha}(\omega^{N-k_U})\tau_V\right\rangle^g$ of (\ref{eq:mainduc1}) factorises into 
\[\prod_{i=1}^n\frac{1}{c_{N+1}(m_i)} \cdot\left\lceil\frac{m_1}{N+1}\right\rceil q^{(N)}_{0,\vec{\kappa}_1,(\beta)}(m_{I_1})q^{(N)}_{g,\vec{\kappa}_2}(m_{I_2})\]
and $\left\lceil\frac{m_1}{N+1}\right\rceil q^{(N)}_{0,\vec{\kappa}_1,(\beta)}(m_{I_1})q^{(N)}_{g,\vec{\kappa}_2}(m_{I_2})$ is a quasi-polynomial of degree at most
\begin{equation}  \label{eq:degpol}
1+(-3+|I_1|+\#\vec{\kappa_1})+(3g-3+|I_2|+\#\vec{\kappa_2})=3g-3+n+s
\end{equation}
since $|I_1|+|I_2|=n$ and $\#\vec{\kappa_1}+\#\vec{\kappa_2}=s+2$.  This also includes the special case involving the genus zero 2-point factor $\left\langle\tau_0(\omega^k)\tau_{m_1-1}(pt)\right\rangle^0=1/c_{N+1}(m_i)$.

Hence the theorem is proven.
\end{proof}
{\em Remarks.}  
1. For $g=1$ a different genus one topological recursion relation 
\[
\left\langle\left\langle \tau_{m_1+1}(\alpha_1)\right\rangle\right\rangle^1=\left\langle\left\langle \tau_{m_1}(\alpha_1)T_j\right\rangle\right\rangle^0\left\langle\left\langle T^j\right\rangle\right\rangle^1 +\frac{1}{12}\left\langle\left\langle T_jT^j \tau_{m_1}(\alpha_1)\right\rangle\right\rangle^0.
\]
can be used to drop the restriction $m_i\geq 3g-1$.  Getzler's genus two topological recursion relation \cite{GetTop} can be used to relax the restriction to $m_i\geq 2$ for $g=2$.

2. Theorem~\ref{th:GWquasi1} is essentially due to the topological recursion relations which are satisfied in general. Thus it might be possible to generalise the theorem to other target spaces, such as Fano manifolds.  The proof can break down at the 2-point function part of the induction.  For example it breaks down for non-stationary Gromov-Witten invariants of $\bp^N$. The following example shows non-polynomial behaviour of the non-stationary invariants.

{\em Example.}  The non-stationary invariants are not quasi-polynomial in the descendant variables.  Consider the case of $\bp^1$ for simplicity.
\begin{align*}
\left\langle\tau_m(1)\tau_0(pt)\tau_0(pt)\right\rangle&=\left\langle\tau_{m-1}(1)\tau_0(1)\right\rangle\left\langle\tau_0(pt)\tau_0(pt)\tau_0(pt)\right\rangle&{\rm TRR}\\
&=\left\langle\tau_{m-2}(1)\right\rangle&{\rm string\ eq}\\
&=\frac{1}{d^2}\left(\left\langle\tau_{m-2}(1)\tau_0(pt)\tau_0(pt)\right\rangle-\left\langle\tau_{m-3}(pt)\right\rangle-\left\langle\tau_{m-3}(pt)\tau_0(pt)\right\rangle\right)&{\rm divisor\ eq}\\
&=\frac{1}{d^2}\left(\left\langle\tau_{m-2}(1)\tau_0(pt)\tau_0(pt)\right\rangle-\frac{1}{c_2(m-1)}-\frac{d-1}{c_2(m-1)}\right)&
\end{align*}
for $d=\left\lceil m/2\right\rceil$ and we have used the 1-point and 2-point stationary formulae.  Define 
\[f(m):=\left\langle\tau_m(1)\tau_0(pt)\tau_0(pt)\right\rangle\cdot c_2(m)\]
then we have proven
\[ f(m)=(1-1/d)f(m-2)-1,\quad d=\left\lceil m/2\right\rceil\]
and clearly $f(m)$ is not quasi-polynomial in $m$.

\subsection{The point $\bp^0$}
The dimension constraint in the $N=0$ case is
\begin{equation}  \label{eq:N=0dim}
3g-3+n+d=\sum_{i=1}^nm_i
\end{equation}
which does not correspond directly to Gromov-Witten invariants with target $\bp^0=\{ pt\}$.  This is because all maps to a point have degree $d=0$ or equivalently the Gromov-Witten invariants of a point, $\left\langle\prod_{i=1}^n\tau_{m_i}\right\rangle_g$, are non-trivial only when $3g-3+n=\sum_{i=1}^nm_i$ which constrains $\{m_i\}$.  This reflects the fact that $\bp^N$ is Fano when $N>0$ and it is Calabi-Yau when $N=0$.

Nevertheless, the proof of Theorem~\ref{th:GWquasi1} still applies to $N=0$ and generates a family of {\em polynomials} (since the $\bmod{\ N+1}$ dependence is no longer a restriction when $N=0$) in unconstrained variables $m_i$.  One introduces a non-trivial degree $d$ in this situation by allowing $d$ extra unlabeled points on a curve.

\begin{defn}
For $d\geq 0$, define
\[ \left\langle\prod_{i=1}^n\tau_{m_i}\right\rangle^g_d
:=\frac{1}{d!}\int_{\overline{\modm}^g_{n+d}} \prod_{i=1}^n\psi_i^{m_i}
=\left\langle\prod_{i=1}^n\tau_{m_i}\cdot\exp\tau_0\right\rangle^g
=\left.\left\langle\left\langle\prod_{i=1}^n\tau_{m_i}\right\rangle\right\rangle^g\right|_{\vec{t}=(1,0,0,...)}.\]
\end{defn}
In particular, $\left\langle\prod_{i=1}^n\tau_{m_i}\right\rangle^g_d$ is non-trivial only when the dimension constraint (\ref{eq:N=0dim}) is satisfied, and as expected $\left\langle\prod_{i=1}^n\tau_{m_i}\right\rangle^g_{d=0}=\left\langle\prod_{i=1}^n\tau_{m_i}\right\rangle^g$.  The generating function is
\[
\left\langle\left\langle\prod_{i=1}^n\tau_{m_i}\right\rangle\right\rangle^g_d:=
\left\langle\prod_{i=1}^n\tau_{m_i}\cdot\exp\tau_0\cdot\exp\sum_{k}t_k\tau_k\right\rangle^g
=\left\langle\prod_{i=1}^n\tau_{m_i}\cdot\exp\sum_{k}\tilde{t}_k\tau_k\right\rangle^g,\quad \tilde{t}_k=t_k+\delta^1_k.
\]
 
\begin{lemma}
$\left\langle\prod_{i=1}^n\tau_{m_i}\right\rangle^g_d$ satisfies the topological recursion relations (\ref{eq:TRR0}) and (\ref{eq:TRRg}).
\end{lemma}
\begin{proof}
This is immediate from the proofs of the topological recursion relations.   The topological recursion relations apply to the target $X=\{ pt\}$ and can be expressed using the generating function $\left\langle\left\langle\prod_{i=1}^n\tau_{m_i}\right\rangle\right\rangle^g$.  Simply substituting $t_0\mapsto t_0+1$ gives the result for the generating function $\left\langle\left\langle\prod_{i=1}^n\tau_{m_i}\right\rangle\right\rangle^g_d$ that includes degree.  We present the proof here.

For $I\subset\{1,...,n\}$, denote by $\overline{\modm}^g_I\cong\overline{\modm}^g_{|I|}$ having points labeled by the subset $I$.  For $I_1\sqcup I_2=\{1,...,n\}$ and $g_1+g_2=g$ define the boundary divisor
\[ D(g_1,I_1|g_2,I_2):=\overline{\modm}^{g_1}_{\{I_1,0\}}\times\overline{\modm}^{g_2}_{\{I_2,0\}}\hookrightarrow\overline{\modm}^{g}_n\]
obtained by gluing the points labeled by $\{0\}$ on each component to obtain a new stable curve.  Define $\widehat{D}(g_1,I_1|g_2,I_2)\in H^2(\overline{\modm}^{g}_n,\bq)$ to be the Poincare dual of $D(g_1,I_1|g_2,I_2)$.
Consider the forgetful map  $\pi_I:\overline{\modm}^g_n\to\overline{\modm}^g_I$ for $I\subset\{1,...,n\}$.  Then for $i\in I$, 
\begin{equation}  \label{eq:boundiv}
\psi_i=\pi_I^*\psi_i\ +\hspace{-1cm}\sum_{\begin{array}{c}I_1\sqcup I_2=\{1,...,n\}\\I_2\cap I=\{i\}\end{array}} \hspace{-1cm}\widehat{D}(g,I_1|0,I_2)
\end{equation}
which follows from the simplest case 
\[ \psi_j=\pi_I^*\psi_j+\widehat{[s_j]},\quad I=\{1,...,n-1\}\] 
and the identification of the image of the $j$th section $[s_j]$ with $D(g,I\backslash\{j\}|0,\{j,n\})$.  See, Getzler \cite{GetTop}.

Since $\psi_i=0$ on $\overline{\modm}^0_3$, the relation (\ref{eq:boundiv}) expresses $\psi_i$ over $\overline{\modm}^0_n$ as a sum of boundary divisors by setting $g=0$ and $|I|=3$ so that $\pi_I^*\psi_i=0$.  Multiply (\ref{eq:boundiv}) by $\prod\psi_i^{m_i}$ and integrate over $\overline{\modm}^g_n$ to get
 \[ \left\langle\left\langle \tau_{m_1+1} \tau_{m_2} \tau_{m_3} \right\rangle\right\rangle^0=\sum_j\left\langle\left\langle \tau_{m_1}\tau_0\right\rangle\right\rangle^0 \left\langle\left\langle \tau_0\tau_{m_2}\tau_{m_3}\right\rangle\right\rangle^0,
 \]
 and further set $\vec{t}=(1+t_0,t_1,t_2,...)$ which shows $\left\langle\prod_{i=1}^n\tau_{m_i}\right\rangle^0_d$ satisfies the topological recursion relation (\ref{eq:TRR0}) as required.
 
The proof that $\left\langle\prod_{i=1}^n\tau_{m_i}\right\rangle^g_d$ satisfies the genus $g$ topological recursion relation (\ref{eq:TRRg}) uses $\psi^{3g-1}=0$ on $\modm^g_1$ together with a pull-back formula for psi classes on the moduli space of {\em pre-stable} curves.  The vanishing of $\psi^{3g-1}$ on $\modm^g_1$ is simply due to the dimension constraint and is applied analogously to the property $\psi_i=0$ used on $\overline{\modm}^0_3$ above.  The moduli space of pre-stable curves includes genus 0 components with only two distinguished components which we will not go into here.  Instead we give a consequence that can be stated over the moduli space of stable curves and found in \cite{GatTop}.
\begin{equation}
 \left\langle\left\langle \tau_{m+3g-1} \right\rangle\right\rangle^g=\sum_{k=1}^{3g-1}(-1)^{k-1}\hspace{-1.2cm}\sum_{\begin{array}{c}(m_0,...,m_k)\\k+\sum m_i=3g-1\end{array}} \hspace{-1cm} \left\langle\left\langle \tau_{m_0} \right\rangle\right\rangle^g
\left\langle\left\langle \tau_0\tau_{m_1} \right\rangle\right\rangle^0\cdots
\left\langle\left\langle \tau_0\tau_{m_{k-1}} \right\rangle\right\rangle^0
\left\langle\left\langle \tau_0\tau_{m_k+m} \right\rangle\right\rangle^0.
\end{equation}
Set $\vec{t}=(1+t_0,t_1,t_2,...)$ and define
\[ \left\langle\left\langle\tau_0\tau_m\right\rangle\right\rangle_{(\beta)}=\sum_{k=1}^{3g-1}(-1)^{k-1}\hspace{-1.2cm}\sum_{\begin{array}{c}(m_1,...,m_k)\\k+\displaystyle{\sum_{i>0} m_i}=\beta+1\end{array}} \hspace{-1cm} 
\left\langle\left\langle \tau_0\tau_{m_1} \right\rangle\right\rangle^0\cdots
\left\langle\left\langle \tau_0\tau_{m_{k-1}} \right\rangle\right\rangle^0
\left\langle\left\langle \tau_0\tau_{m_k+m} \right\rangle\right\rangle^0.
\]
to prove that $\left\langle\prod_{i=1}^n\tau_{m_i}\right\rangle^g_d$ satisfies the genus $g$ topological recursion relation (\ref{eq:TRRg}). 
\end{proof}

Recall that $c_1(m)=m!$ so in the $N=0$ case (\ref{eq:quasi}) is given by
\begin{equation}   \label{eq:GWpt}
p_g(m_1,...,m_n)=p^{(0)}_g(m_1,...,m_n):=\prod_{i=1}^nm_i!\left\langle\prod_{i=1}^n\tau_{m_i}\cdot\exp\tau_0\right\rangle^g
\end{equation}
where we drop the superscript $(0)$.  Each $p_g(m_1,...,m_n)$ is a polynomial, by Theorem~\ref{th:GWquasi1}, or more directly from the following lemma which gives an explicit formula and will be needed later.
\begin{lemma}  \label{th:formula0}
\[
p_g(m_1,...,m_n)=\prod_{i=1}^nm_i!\left\langle\prod_{i=1}^n\tau_{m_i}\cdot\exp\tau_0\right\rangle^g
=\sum_{|\beta|=3g-3+n\ }\prod_{i=1}^n\binom{m_i}{\beta_i}\cdot\beta_i!
\left\langle\prod_{i=1}^n\tau_{\beta_i}\right\rangle^g.
\]
\end{lemma}
\begin{proof}
There are three cases corresponding to
\[ d=\sum m_i-(3g-3+n)\]
being negative, zero and positive.  

When $d<0$, the formula is true since both sides vanish.  The left hand side $\left\langle\prod_{i=1}^n\tau_{m_i}\cdot\exp\tau_0\right\rangle^g=0$ by dimensional constraints.  The right hand side vanishes since $d<0$ implies there exists an $i\in\{1,...,n\}$ such that $m_i<\beta_i$ and hence the factor $\binom{m_i}{\beta_i}$ vanishes.

When $d=0$, the left hand side is simply $\prod_{i=1}^nm_i!\left\langle\prod_{i=1}^n\tau_{m_i}\right\rangle^g$.  On the right hand side, unless $\beta_i=m_i$ for all $i=1,...,n$, then there exists an $i\in\{1,...,n\}$ such that $\beta_i>m_i$ and hence the factor $\binom{m_i}{\beta_i}$ vanishes.  This leaves the only surviving summand $\prod_{i=1}^nm_i!\left\langle\prod_{i=1}^n\tau_{m_i}\right\rangle^g$ where $\beta_i=m_i$ for all $i=1,...,n$.

When $d>0$, then $\exp\tau_0$ makes a non-trivial contribution, so in particular there is a $\tau_0$ term and $p_g(m_1,...,m_n)$ is uniquely determined by the string equation and the $d=0$ case.  The string equation is
\[ d\cdot p_g(m_1,...,m_n)=\sum_{i=1}^nm_i\cdot p_g(m_1,...,m_i-1,...,m_n)
\]
and the formula is proven by induction on $d$.  The initial case $d=0$ has been proven.  For $d>0$,
\begin{align*}
d\cdot p_g(m_1,...,m_n)&=\sum_{i=1}^nm_i\cdot p_g(m_1,...,m_i-1,...,m_n)& {\rm string\ equation}\\
&=\sum_{i=1}^n\hspace{0cm}\sum_{|\beta|=3g-3+n\ }\hspace{-.5cm}m_i\cdot\binom{m_1}{\beta_1}\cdot\cdot\binom{m_i-1}{\beta_i}\cdot\cdot\binom{m_n}{\beta_n}\cdot\prod\beta_i!
\left\langle\prod_{i=1}^n\tau_{\beta_i}\right\rangle^g&{\rm inductive\ hypothesis}\\
&=\hspace{-.5cm}\sum_{|\beta|=3g-3+n\ }\hspace{-0cm}\sum_{i=1}^n(m_i-\beta_i)\cdot\prod\binom{m_i}{\beta_i}\cdot\beta_i!
\left\langle\prod_{i=1}^n\tau_{\beta_i}\right\rangle^g&\\
&=d\cdot\hspace{-.3cm}\sum_{|\beta|=3g-3+n\ }\hspace{-.2cm}\prod\binom{m_i}{\beta_i}\cdot\beta_i!
\left\langle\prod_{i=1}^n\tau_{\beta_i}\right\rangle^g&\Leftarrow\quad\sum(m_i-\beta_i)=d\\
\end{align*}
and divide both sides by $d>0$ to get the result.
\end{proof}

\subsection{Asymptotic behaviour of Gromov-Witten invariants}
The large degree asymptotic behaviour of stationary Gromov-Witten invariants of $\bp^N$ is stated in Corollary~\ref{th:asym}.  It is an immediate consequence of Theorem~\ref{th:GWquasi} which gives the highest degree terms in the polynomial part.  These highest degree terms, and hence the asymptotic behaviour, are governed by the $N=0$ case.

\begin{proof}[Proof of Theorem~\ref{th:GWquasi}]
The proof of quasi-polynomiality follows immediately from Theorem~\ref{th:GWquasi1}.  It remains to prove that the top degree terms are in fact {\em polynomial}, i.e. there is no $\bmod{\ N+1}$ dependence, with coefficients given by intersections of $\psi$ classes over the moduli space of curves.

We first prove that the top degree terms are essentially independent of $N$--- the dependence is simply the factor $(N+1)^{3-2g-n}$.  The proof is, as usual, by induction.  The recursions that define the quasi-polynomials term-by-term correspond between different $N$.  The quasi-polynomials contain constants that depend on Gromov-Witten invariants of insertions $\tau_m$ for $m<3g-1$.  But the top degree terms are independent of these constants---the two degree terms are both strictly positive in (\ref{eq:degpol}) in the construction of the quasi-polynomials contained in the proof of Theorem~\ref{th:GWquasi1}.  By the inductive hypothesis the top degree term of each quasi-polynomial is $(N+1)^{3-2g'-n'}$ times an intersection number, independent of $N$.    The genus zero 2-point invariants contribute only a factor of 1 to the more complicated quasi-polynomial so have no influence.

The independence of the top degree terms from $N$ immediately implies the top degree terms are polynomial by considering the $N=0$ case, i.e. the top degree terms of the $N=0$ polynomials coincide (up to a factor of $(N+1)^{3-2g-n}$) with the top degree terms of the quasi-polynomials from $\bp^N$, $N>0$.  In particular there is no $\bmod{N+1}$ dependence in the top degree terms.

To prove that the top coefficients are intersection numbers of $\psi$ classes over the moduli space of curves it is enough to prove this for any $N$.  Again we consider the $N=0$ case.  We use the explicit formula for $p_g(m_1,...,m_n)$ given in Lemma~\ref{th:formula0}.  For $\beta_i$ constant and $m_i$ a variable, the top coefficient of the polynomial $\binom{m_i}{\beta_i}\cdot\beta_i!$ is $m_i^{\beta_i}$.  Hence the top coefficients $c_{\beta}$ of $m_1^{\beta_1}\cdots m_n^{\beta_n}$ in $p_g(m_1,...,m_n)$ are
\[c_{\beta}=\left\langle\prod_{i=1}^n\tau_{\beta_i}\right\rangle^g=\int_{\overline{\modm}_{g,n}}\psi_1^{\beta_1}...\psi_n^{\beta_n}\]
as required.  For $N>0$, the top coefficients are
\[c^{(N)}_{\beta}=(N+1)^{3-2g-n}\int_{\overline{\modm}_{g,n}}\psi_1^{\beta_1}...\psi_n^{\beta_n}\]
and the theorem is proven.
\end{proof}
{\em Remark.}  The $N=1$ case has already been studied in \cite{NScGro} where the coefficients are proven to be the correct intersection numbers of $\psi$ classes when $g=0$ and $g=1$.

It will be useful to identify the quasi-polynomials $q^{(N)}_{g,\vec{\kappa}}$ in Theorem~\ref{th:GWquasi1} with the stationary quasi-polynomials $p^{(N)}_g$ evaluated at particular values.   
\begin{theorem}  \label{th:negeval2}
When all non-stationary insertions are primary, $m_i\geq 3g-1$, $i=1,...,n$ and $n+s>2$
\[q^{(N)}_{g,\vec{\kappa}}(m_1,\dots,m_n)=p^{(N)}_g(k_1-N,...,k_s-N,m_1,\dots,m_n)\] 
where $\vec{\kappa}=(0,k_1,0,...,0,k_s)$.
\end{theorem}
\begin{proof}
We start with $n=0$, where we need to prove that primary invariants are stored in the stationary quasi-polynomials.

{\bf Genus 0.}
Using the genus zero TRR (\ref{eq:TRR0}),
\[
\left\langle \prod_{j=1}^s\tau_{n_j}(pt)\right\rangle^0=\left\langle \tau_{n_1-1}(pt)\tau_0(\omega^k)\right\rangle^0\left\langle \tau_0(\omega^{N-k})\prod_{j=2}^s\tau_{n_j}(pt)\right\rangle^0+\left\langle \tau_{n_1-1}(pt)...\right\rangle^0
\]
where the RHS contains one summand, which is shown, with $\tau_{n_1-1}(pt)$ in a genus zero 2-point invariant, and all other summands with $\tau_{n_1-1}(pt)$ in a genus zero $k$-point invariant for $k>2$.  Hence
\begin{align*}
\left\langle \prod_{j=1}^s\tau_{n_j}(pt)\right\rangle^0c_{N+1}(n_1)
&=\left\langle \tau_{n_1-1}(pt)\tau_0(\omega^{N-k_1})\right\rangle^0c_{N+1}(n_1)\left\langle \tau_0(\omega^{k_1})\prod_{j=2}^s\tau_{n_j}(pt)\right\rangle^0+\left\lceil\frac{n_1}{N+1}\right\rceil q(n_1,..,n_m)\\
&=\left\langle \tau_0(\omega^{k_1})\prod_{j=2}^s\tau_{n_j}(pt)\right\rangle^0+\left\lceil\frac{n_1}{N+1}\right\rceil\times q(n_1,...,n_m)
\end{align*}
where $\left\langle \tau_{n_1-1}(pt)\tau_0(\omega^k)\right\rangle^0\cdot c_{N+1}(n_1)=1$ from Lemma~\ref{th:2primdes} and $q(n_1,...,n_m)$ is quasi-polynomial in $n_1$ from Theorem~\ref{th:GWquasi1}.   The factor $\left\lceil\frac{n_1}{N+1}\right\rceil$ comes from $c_{N+1}(n_1)=\left\lceil\frac{n_1}{N+1}\right\rceil c_{N+1}(n_1-1)$ and is explicit in the proof of Theorem~\ref{th:GWquasi1}.  Dimension considerations yield $k_1\equiv m_1+N\bmod{N+1}$.

Now apply the genus zero TRR (\ref{eq:TRR0}) to $\left\langle \tau_0(\omega^{k_1})\prod_{j=2}^s\tau_{n_j}(pt)\right\rangle^0$ and further apply it iteratively to get
\begin{equation}  \label{eq:eval}
\left\langle \prod_{j=1}^s\tau_{n_j}(pt)\right\rangle^0\cdot \prod_{j=1}^sc_{N+1}(n_j)=\left\langle\prod_{j=1}^s \tau_0(\omega^{k_j})\right\rangle^0+\sum\left\lceil\frac{n_j}{N+1}\right\rceil\times q_j(n_1,...,n_m).
\end{equation}
Combine (\ref{eq:eval}) with
\[ \left\lceil\frac{k_j-N}{N+1}\right\rceil=0,\quad k_j\in\{0,...,N\}\]
and
\[p^0_s(n_1,...,n_s)=\left\langle \prod_{j=1}^s\tau_{n_j}(pt)\right\rangle^0\cdot\prod_{j=1}^sc_{N+1}(n_j)\]
to prove
\[  \left\langle \prod_{j=1}^s\tau_0(\omega^{k_j})\right\rangle^0= p^0_{s}(k_1-N,...,k_s-N)\]
which is the genus zero $n=0$ case.

We now reduce the genus zero $n>0$ case to the $n=0$ case using induction and the TRR (\ref{eq:TRR0}).  

We need to compare the two expressions
$\left\langle \prod_{j=1}^s\tau_0(\omega^{k_j})\prod_{i=1}^n\tau_{m_i}(pt)\right\rangle^g$
and
$\left\langle \prod_{j=1}^s\tau_{n_j}(pt)\prod_{i=1}^n\tau_{m_i}(pt)\right\rangle^g$
for $n_j\equiv k_j-N\bmod{N+1}$.
We can assume $m_1>0$ since it is a variable.  Since $n+s>2$ we can apply the genus zero TRR (\ref{eq:TRR0}) to get
\begin{equation}  \label{eq:mainduc2}
\left\langle \prod_{j=1}^s\tau_0(\omega^{k_j})\prod_{i=1}^n\tau_{m_i}(pt)\right\rangle^0
=\sum_{U}\left\langle\tau_0(\omega^{k_U})\tau_{m_1-1}(pt)\tau_U\right\rangle\left\langle\tau_0(\omega^{N-k_U})\tau_{a_2}(\omega^{b_2})\tau_{a_3}(\omega^{b_3})\tau_V\right\rangle.
\end{equation}
\begin{equation}  \label{eq:mainduc3}
\left\langle \prod_{j=1}^s\tau_{n_j}(pt)\prod_{i=1}^n\tau_{m_i}(pt)\right\rangle^0
=\sum_{U}\left\langle\tau_0(\omega^{k_U})\tau_{m_1-1}(pt)\tau'_{U}\right\rangle\left\langle\tau_0(\omega^{N-k_U})\tau_{a_2}(pt)\tau_{a_3}(pt)\tau'_V\right\rangle.
\end{equation}
where $\tau_{m_1}(pt)\tau_{a_2}(\omega^{b_2})\tau_{a_3}(\omega^{b_3})\tau_U\tau_V=\prod_{j=1}^s\tau_0(\omega^{k_j})\prod_{i=1}^n\tau_{m_i}(pt)$ together with corresponding  factorisations $\tau_{m_1}(pt)\tau_{a_2}(pt)\tau_{a_3}(pt)\tau'_U\tau'_V=\prod_{j=1}^s\tau_{n_j}(pt)\prod_{i=1}^n\tau_{m_i}(pt)$.  By the inductive assumption, corresponding terms on the right hand sides coincide when we evaluate the quasi-polynomial at $n_j=k_j-N$, and remove the factor $1/c(n_j)$.  The condition $n_j\equiv N-k_j\bmod{N+1}$ determines which polynomial representative of the quasi-polynomial to take (and guarantees the $k_U$ coincide.)

The initial conditions for the induction are the case $n=0$.  Note that the 2-point invariants are not initial conditions, as in the inductive proof of Theorem~\ref{th:GWquasi1}, since they appear here as the same factor $\left\langle\tau_{m_1-1}(pt)\tau_0(\omega^k)\right\rangle^0$ trivially equal in the two expressions.  Hence we have reduced the case $n>0$ to the case $n=0$ and the Theorem is proven for genus zero.\\

{\bf Genus $\bf g>0$.}  From the description of the genus 0 expression $\left\langle \tau_{n_1+1-3g}(pt)\tau_0(\omega^k)\right\rangle_{(\beta)}$ in the proof of Theorem~\ref{th:GWquasi1} we see that for $\beta=3g-2$ we have
\[ \left\langle \tau_{n_1+1-3g}(pt)\tau_0(\omega^k)\right\rangle_{(3g-2)}
=\left\langle \tau_{n_1-1}(pt)\tau_0(\omega^{N-k_1})\right\rangle^0+\left\langle \tau_r(pt)\cdots\right\rangle^0,\quad r<n_1-1.\]
Combining this with the genus $g$ TRR (\ref{eq:TRRg}),
\[
\left\langle \prod_{j=1}^s\tau_{n_j}(pt)\right\rangle^g=\left\langle \tau_{n_1+1-3g}(pt)\tau_0(\omega^k)\right\rangle_{(3g-2)}\left\langle \tau_0(\omega^{N-k})\prod_{j=2}^s\tau_{n_j}(pt)\right\rangle^g+\left\langle \tau_{n_1+1-3g}(pt)...\right\rangle
\]
where the RHS contains one summand, which is shown, with $\tau_{n_1+1-3g}(pt)$ in a genus zero 2-point invariant, and all other summands with $\tau_{n_1+1-3g}(pt)$ in a $(g',n')$-point invariant for $2g'-2+n'>0$, we get
\begin{align*}
\left\langle \prod_{j=1}^s\tau_{n_j}(pt)\right\rangle^gc_{N+1}(n_1)
&=\left\langle \tau_{n_1-1}(pt)\tau_0(\omega^{N-k_1})\right\rangle^0c_{N+1}(n_1)\left\langle \tau_0(\omega^{k_1})\prod_{j=2}^s\tau_{n_j}(pt)\right\rangle^g+\left\lceil\frac{n_1}{N+1}\right\rceil q(n_1,..,n_s)\\
&=\left\langle \tau_0(\omega^{k_1})\prod_{j=2}^s\tau_{n_j}(pt)\right\rangle^g+\left\lceil\frac{n_1}{N+1}\right\rceil\times q(n_1,...,n_s).
\end{align*}
Again we have used $\left\langle \tau_{n_1-1}(pt)\tau_0(\omega^k)\right\rangle^0\cdot c_{N+1}(n_1)=1$ from Lemma~\ref{th:2primdes} and we have $q(n_1,...,n_s)$ quasi-polynomial in $n_1$ from Theorem~\ref{th:GWquasi1}.  This iteratively leads to an expression analogous to (\ref{eq:eval}) for $g>0$:
\[
\left\langle \prod_{j=1}^s\tau_{n_j}(pt)\right\rangle^g\cdot \prod_{j=1}^sc_{N+1}(n_j)=\left\langle\prod_{j=1}^s \tau_0(\omega^{k_j})\right\rangle^g+\sum\left\lceil\frac{n_j}{N+1}\right\rceil\times q_j(n_1,...,n_m)
\]
and hence 
\[  \left\langle \prod_{j=1}^s\tau_0(\omega^{k_j})\right\rangle^g= p^{(N)}_g(k_1-N,...,k_s-N)\]
as required.

The reduction of the $g>0$, $n>0$ case to the $g>0$, $n=0$ uses the genus $g$ TRR (\ref{eq:TRRg} which yield (\ref{eq:mainduc1}).   The right hand side of (\ref{eq:mainduc1}) is simpler in the induction---either it is genus 0 or there are fewer than $n+s$ insertions in each factor, or there are fewer than $n$ descendant insertions.  Hence evaluation at negative values reduces to the proven initial cases of evaluation at negative values for the genus zero quasi-polynomials or to the $n=0$ case.
\end{proof}
\begin{proof}[Proof of Theorem~\ref{th:negeval}]
The genus zero 2-point case can be checked explicitly using Lemma~\ref{th:2primdes} which gives
\[\langle\tau_{m}(pt)\tau_0(\omega^k)\rangle=\frac{1}{c_{N+1}(m)}\cdot \frac{1}{d},\quad d=\frac{m+k+1}{N+1}\]
and
\[\langle\tau_{m_1}(pt)\tau_{m_2}(pt)\rangle=\frac{1}{c_{N+1}(m_1)c_{N+1}(m_2)}\cdot \frac{1}{d},\quad d=1+\frac{m_1+m_2}{N+1}.\]
Hence $p^0_2(m_1,m_2)=\frac{N+1}{m_1+m_2+N+1}$ and
\[\langle\tau_{m_1}(pt)\tau_{0}(\omega^k)\rangle\cdot c_{N+1}(m_1)=p^0_2(m_1,k-N)\]
as required.

The remaining cases of Theorem~\ref{th:negeval} are proven in Theorem~\ref{th:negeval2} in terms of the quasi-polynomials.   
\end{proof}

\section{The projective line $\bp^1$}  \label{sec:P1}

One of the earliest studies of stationary Gromov-Witten invariants was undertaken by Okounkov and Pandharipande in their papers \cite{OPaEqu,OPaGro,OPaVir} on Gromov-Witten theory of target curves.  They showed that considering stationary invariants alone is a natural problem, related to Hurwitz problems and partitions.  The case of target $\bp^1$ is fundamental to all of their results.  In \cite{NScGro} the stationary Gromov-Witten invariants of $\bp^1$ were related to Eynard-Orantin invariants which arise out of matrix models but can be defined independently of matrix models.

\subsection{Eynard-Orantin invariants.}  \label{sec:EO}

The Eynard-Orantin invariants are defined for any $(C,x,y)$ consisting of a rational curve $C$ equipped with two meromorphic functions $x$ and $y$ with the property that the branch points of $x$ are simple and the map 
\[ \begin{array}[b]{rcl} C&\to&\bc^2\\p&\mapsto& (x(p),y(p))\end{array}\]
is an immersion.  For every $(g,n)\in\bz^2$ with $g\geq 0$ and $n>0$ the Eynard-Orantin invariant is  a multidifferential $\omega^g_n(p_1,...,p_n)$, i.e. a tensor product of meromorphic 1-forms on the product $C^n$, where $p_i\in C$.   When $2g-2+n>0$, $\omega^g_n(p_1,...,p_n)$ is defined recursively in terms of local  information around the poles of $\omega^{g'}_{n'}(p_1,...,p_n)$ for $2g'+2-n'<2g-2+n$.  The $\omega^{g'}_{n'}(p_1,...,p_n)$ are used as kernels on the Riemann surface.   For example, $\omega^0_2(w,z)$ is the Cauchy kernel $dwdz/(w-z)^2$.

Since each branch point $\alpha$ of $x$ is simple, for any point $p\in C$ close to $\alpha$ there is a unique point $\hat{p}\neq p$ close to $\alpha$ such that $x(\hat{p})=x(p)$.  The recursive definition of $\omega^g_n(p_1,...,p_n)$ uses only local information around branch points of $x$ and makes use of the well-defined map $p\mapsto\hat{p}$ there. The invariants are defined as follows.
\begin{align*}
\omega^0_1&=-ydx(z)\nonumber\\
% \label{eq:berg}
\omega^0_2&=\frac{dz_1dz_2}{(z_1-z_2)^2}
\end{align*}
For $2g-2+n>0$,
\begin{equation}  \label{eq:EOrec}
\omega^g_{n+1}(z_0,z_S)=\sum_{\alpha}\hspace{-2mm}\res{z=\alpha}K(z_0,z)\hspace{-.5mm}\biggr[\omega^{g-1}_{n+2}(z,\hat{z},z_S)+\hspace{-5mm}\displaystyle\sum_{\begin{array}{c}_{g_1+g_2=g}\\_{I\sqcup J=S}\end{array}}\hspace{-5mm}
\omega^{g_1}_{|I|+1}(z,z_I)\omega^{g_2}_{|J|+1}(\hat{z},z_J)\biggr]
\end{equation}
where the sum is over branch points $\alpha$ of $x$, $S=\{1,...,n\}$, $I$ and $J$ are non-empty and 
\[\displaystyle K(z_0,z)=\frac{-\int^z_{\hat{z}}\omega^0_2(z_0,z')}{2(y(z)-y(\hat{z}))dx(z)}\] is well-defined in the vicinity of each branch point of $x$.   Note that the quotient of a differential by the differential $dx(z)$ is a meromorphic function.  The recursion (\ref{eq:EOrec}) depends only on the meromorphic differential $ydx$ and the map $p\mapsto\hat{p}$ around branch points of $x$.  For $2g-2+n>0$, each $\omega^g_n$ is a symmetric multidifferential with poles only at the branch points of $x$, of order $6g-4+2n$, and zero residues.

When $y$ is not a meromorphic function on $C$ and is merely analytic in a domain containing the branch points of $x$, we approximate it by a sequence of meromorphic functions $y^{(N)}$ which agree with $y$ at the branch points of $x$ up to the $N$th derivatives.   For example, define 
\begin{equation}  \label{eq:lnz}
 C=\begin{cases}x=z+1/z\\ 
y=\ln{z}\sim\sum\frac{(1-z^2)^k}{\hspace{-3mm}-2k}
\end{cases}
\end{equation}
via partial sums $y_N$.  The Riemann surface $C$ is defined via the meromorphic function $x(z)$.  The function $y(z)=\ln{z}\sim\sum\frac{(1-z^2)^k}{\hspace{-3mm}-2k}$ is to be understood as the sequence of partial sums $y_N=\displaystyle{\sum_1^N}$$\frac{(1-z^2)^k}{\hspace{-3mm}-2k}$.  Each invariant requires only a finite $y_N$---for fixed $(g,n)$ the sequence of invariants $\omega^g_n$ of $(C,x,y_N)$ stabilises for $N\geq 6g-6+2n$.
As in the introduction assemble the stationary Gromov-Witten invariants of $\bp^1$ into the generating function multidifferential
\[\Omega^g_n(x_1,...,x_n)=\sum_{\vec{m}}\left\langle \prod_{i=1}^n \tau_{m_i}(pt) \right\rangle^g_{\bp^1}\hspace{-.2cm}\cdot\hspace{.2cm}\prod_{i=1}^n(m_i+1)!x_i^{-m_i-2}dx_i.\]
\begin{theorem}[Norbury-Scott]   \label{th:main}
For $g=0$ and 1 and $2g-2+n>0$, the Eynard-Orantin invariants of the curve $C$ defined in (\ref{eq:lnz}) agree with the generating function for the Gromov-Witten invariants of $\bp^1$:
\[ \omega^g_n\sim\Omega^g_n(x_1,...,x_n).\]
More precisely, $\Omega^g_n(x_1,...,x_n)$ gives an analytic expansion of $\omega^g_n$ around a branch of $\{x_i=\infty\}$.
\end{theorem}
In the two exceptional cases $(g,n)=(0,1)$ and $(0,2)$, the invariants $\omega^g_n$ are not analytic at $x_i=\infty$.  We can again get analytic expansions around a branch of $\{x_i=\infty\}$ by removing their singularities at $x_i=\infty$ as follows:  
\begin{equation}  \label{eq:excep} 
\omega^0_1+\ln{x_1}dx_1\sim\Omega^0_1(x_1),\quad\omega^0_2-\displaystyle\frac{dx_1dx_2}{(x_1-x_2)^2}\sim\Omega^0_2(x_1,x_2).
\end{equation}

\begin{proof}[Proof of Theorem~\ref{th:p1gen}]
For $I\subset\{1,...,n\}$ put
\[\Omega^g_{n,I}(x_1,...,x_n)\hspace{.2cm}=\hspace{-1cm}\sum_{\begin{array}{c}m_i\geq 3g-1,\ i\in I\\m_i<3g-1,\ i\not\in I\end{array}}\hspace{-.2cm}\left\langle \prod_{i=1}^n \tau_{m_i}(pt) \right\rangle^g_{\bp^1}\hspace{-.2cm}\cdot\hspace{.2cm}\prod_{i=1}^n(m_i+1)!x_i^{-m_i-2}dx_i\]
so
\[ \Omega^g_n(x_1,...,x_n)=\sum_{I\subset\{1,...,n\}}\Omega^g_{n,I}(x_1,...,x_n).\]
Each $\Omega^g_{n,I}(x_1,...,x_n)$ is polynomial in $x_i^{-1}$, so in particular meromorphic in $x_i$, for $i\not\in I$ .  For $i\in I$, the coefficient of $x_i^{-m_i-2}$ is $(m_i+1)!/c_2(m_i)$ times a quasi-polynomial in $m_i$.  Put $x_i=z_i+1/z_i$.  Direct calculation, see \cite{NScGro}, shows that $\Omega^g_{n,I}(x_1,...,x_n)$ is a rational function in $z_i$ with poles only at $z_i=\pm 1$, or equivalently $x_i=\pm 2$, for $i\in I$.  Hence $\Omega^g_{n,I}(x_1,...,x_n)$ is a rational function in $z_i$ for all $i\in\{1,...,n\}$ and the same is true for $\Omega^g_n(x_1,...,x_n)$.

The order of the poles at $z=\pm 1$ and the asymptotic behaviour there follows from the top degree coefficients of the quasi-polyniomials given by (\ref{eq:gwcoeff}) in Theorem~\ref{th:GWquasi}.  Only the term $\Omega^g_{n,I}(x_1,...,x_n)$ for $I=\{1,...,n\}$ contributes because a proper subset $I\subset\{1,...,n\}$ does not have highest degree terms, or equivalently shifts the poles from $x_i=\pm 2$ to $x_i=0$, for $i\not\in I$.  Thus (\ref{eq:gwcoeff}) immediately gives the asymptotic behaviour for $z_i=\pm1+s\cdot t_i$
\[
\Omega^g_n\sim s^{6-6g-3n}\frac{1}{2^{5g-5+2n}}\sum_{|\beta|}\prod_{i=1}^n \frac{(2\beta_i+1)!}{\beta_i!}\frac{dt_i}{t_i^{2\beta_i+2}}\langle \tau_{\beta_1}...\tau_{\beta_n}\rangle_g
\]
as required.
\end{proof}

\subsection{String and divisor equations}
The Eynard-Orantin invariants quite generally satisfy a string equation for $2g-2+n>0$:
\begin{equation}
\sum_{\alpha}\res{z=\alpha}y(z)x(z)^m\omega^g_{n+1}(z_S,z)=-\sum_{i=1}^n\partial_{z_i}\left(\frac{x(z_i)^m\omega^g_n(z_S)}{dx(z_i)}\right),\indent ~~~m=0,1\label{EOstring}
\end{equation}
where the sum is over the branch points $\alpha$ of $x$, $\Phi(z)=\int^z ydx(z')$ is an arbitrary antiderivative and $z_S=(z_1,\dots,z_n)$.

In the case of $x=z+1/z$, $y=\ln{z}$ this gives a relationship between the stationary Gromov-Witten of $\bp^1$.   {\em A priori} the string equations satisfies by Eynard-Orantin invariants and Gromov-Witten invariants should not be related since the latter involves non-stationary terms.  The string equations do in fact coincide when we interpret the non-stationary term in the string equation as a stationary term.  Theorem~\ref{th:negeval} does exactly that more generally.  The following is a generalisation of this observation to projective spaces.

Gromov-Witten invariants satisfy the divisor and string equations which necessarily involve non-stationary terms.
\begin{align*}
\left\langle \tau_0(\omega)\prod_{i=1}^n\tau_{m_i}(pt)\right\rangle^g_d&=d\left\langle \prod_{i=1}^n\tau_{m_i}(pt)\right\rangle^g_d&{\rm divisor\ equation}\\
\left\langle \tau_0(1)\prod_{i=1}^n\tau_{m_i}(pt)\right\rangle^g_d&=\sum_{i=1}^n\left\langle \tau_{m_1}(pt)\cdots\tau_{m_i-1}(pt)\cdots \tau_{m_n}(pt)\right\rangle^g_d&{\rm string\ equation}.
\end{align*}
where the term $\tau_{m_i-1}(\gamma_1)$ vanishes if $m_i=0$ and 
$d=d(m_1,...,m_n)$ is defined by the dimension constraint $(N+1)d+(N-3)(1-g)+n=\sum_{i=1}^n m_i+nN$ .  These give rise to divisor and string equations, respectively, between the quasi-polynomials:
\[ p^{(N)}_g(1-N,m_1,...,m_n)=d\cdot p^{(N)}_g(m_1,...,m_n)\]
and
\[ p^{(N)}_g(-N,m_1,...,m_n)=\sum_{i=1}^n\left\lceil\frac{m_i}{N+1}\right\rceil p^{(N)}_g(m_1,...,m_i-1,...,m_n).\]
These are relations between stationary invariants meaning that both sides are determined by stationary invariants alone.  For example, when $N=1$ these can be used to uniquely determine the genus 0 stationary invariants.  

The Eynard-Orantin invariants satisfy a dilaton equation
\[\sum_{\alpha}\res{z=\alpha}\Phi(z)\omega^g_{n+1}(z_S,z)=(2g-2+n)\omega^g_n(z_S)\]
where $\Phi$ is defined by $d\Phi=ydx$.  The Gromov-Witten invariants also satisfy a dilaton equation
\[\left\langle \tau_1(1)\prod_{i=1}^n\tau_{m_i}(pt)\right\rangle^g_d=(2g-2+n)\left\langle \prod_{i=1}^n\tau_{m_i}(pt)\right\rangle^g_d.\]
Again this involves a non-stationary term, while the Eynard-Orantin recursion suggests it can be expressed in terms of stationary invariants.  Indeed, in the $N=1$ case it was shown in \cite{NScGro} that for $g=0$ or $1$ and conjecturally for all $g$ that $\tau_1(1)$ classes can be evaluated via the derivative of the quasi-polynomial:
\begin{equation}
\left\langle\tau_1(1) \prod_{i=1}^n\tau_{m_i}(pt)\right\rangle^g_{\bp^1}\hspace{-.2cm}\cdot\hspace{.2cm}\prod_{i=1}^n c_2(m_i)=2\frac{\partial}{\partial m_{n+1}}p_g(m_1,\dots,m_n,m_{n+1})\Big|_{m_{n+1}=0}.
\end{equation}
A similar recursion may hold for general $N$, or perhaps not and instead this recursion may be indicative of a relation to Eynard-Orantin invariants which only occurs for $N=1$.

\section{Examples}  \label{sec:ex}

\begin{center}
\begin{spacing}{2.5}  
\begin{tabular}{||c|c|c|c||} 
\hline\hline

{\bf g} &{\bf n}&$\bf p^{(N)}_g(m_1,\dots,m_n)$& {\bf top degree terms of} $\bf p^{(N)}_g$\\ \hline
0&2&$\displaystyle\frac{N+1}{(m_1+m_2+N+1)}$&$\displaystyle\frac{N+1}{(m_1+m_2)}$\\ \hline
0&3&$1$&$1$\\ \hline
1&1&$\displaystyle\frac{N+1}{24}\left\lceil\frac{m}{N+1}\right\rceil+\langle \tau_{0}(\omega)\rangle^1$&$\displaystyle\frac{m}{24}$\\ \hline
0&4& $\displaystyle\sum_{i=1}^4\left\lceil\frac{m_i}{N+1}\right\rceil+\left\langle \prod_{i=1}^4 \tau_0(\omega^{\overline{m_i}})\right\rangle^0$&$\displaystyle\frac{1}{N+1}\sum_{i=1}^4m_i$\\ \hline
1&2&$\displaystyle \frac{N+1}{24}\sum_{i=1}^2\left\lceil\frac{m_i}{N+1}\right\rceil\left\lceil\frac{m_i-1}{N+1}\right\rceil+
\left\lceil\frac{m_1}{N+1}\right\rceil\left\lceil\frac{m_2}{N+1}\right\rceil$&$\displaystyle\frac{m_1^2+m_2m_2+m_2^2}{24(N+1)}$
\\
&&$+\displaystyle\sum_{i\neq j}\left\langle\tau_0(\omega^{\overline{m_i-1}})\tau_1(\omega^{\overline{m_j}})\right\rangle^1\left\lceil\frac{m_i}{N+1}\right\rceil+\left\langle\prod_{i=1}^2\tau_0(\omega^{\overline{m_i}})\right\rangle^1$&\\
\hline\hline
\end{tabular} 
\end{spacing}
\end{center}
%\end{table}
where $\overline{m}\equiv m+N\bmod{N+1}$ and $0\leq\overline{m}\leq N$.  Explicit formulae for the Gromov-Witten invariants are:
\begin{itemize}
\item{Genus zero 2-point invariants:
 \begin{align*}
\langle \tau_{m_1}(pt)\tau_{m_2}(pt)\rangle^{g=0} &=\frac{1}{c_{N+1}(m_1)c_{N+1}(m_2)} \cdot\frac{N+1}{(m_1+m_2+N+1)}
\end{align*}}
\item{Genus zero 3-point invariants: 
\begin{align*}
\langle \tau_{m_1}(pt)\tau_{m_2}(pt)\tau_{m_3}(pt)\rangle^{g=0} &=\frac{1}{c_{N+1}(m_1)c_{N+1}(m_2)c_{N+1}(m_3)} 
\end{align*}}
\item{Genus zero 4-point invariants:
\begin{align*}
\left\langle \prod_{i=1}^4 \tau_{m_i}(pt)\right\rangle^{g=0} &=\prod_{i=1}^4\frac{1}{c_{N+1}(m_i)}\cdot\left(\frac{1}{N+1}\sum_{i=1}^4 m_i+r(m_1,m_2,m_3,m_4) \right)
\end{align*}}
\item{Genus one 1-point invariants:
 \begin{equation*}
\langle \tau_{m}(pt)\rangle^1 =\frac{1}{c_{N+1}(m)}\cdot\left(\frac{m}{24}+r(m)\right)
\end{equation*}}
\item{Genus one 2-point invariants:
 \begin{equation*}
\langle \tau_{m_1}(pt)\tau_{m_2}(pt)\rangle^1 =\prod_{i=1}^2\frac{1}{c_{N+1}(m_i)}\cdot\left(\frac{m_1^2+m_2m_2+m_2^2}{24(N+1)}+\sum_{i=1}^2r_i(m_1,m_2)m_i+r(m_1,m_2)\right)
\end{equation*}}
\end{itemize}
where $r(\vec{m})$ and $r_j(\vec{m})$ depend only on the residue class of $m_i\bmod{N+1}$.

\end{document}